\documentclass[a4paper,fleqn]{article}
\usepackage{amssymb,amsthm,amsfonts}
\usepackage{amsmath}
\newtheorem{theorem}{Theorem}

\numberwithin{equation}{section}
\numberwithin{lemma}{section}
\numberwithin{theorem}{section}
\numberwithin{corollary}{section}

\allowdisplaybreaks
\begin{document}
\setcounter{page}{1}

\title{Some  results on  the Kamp\'e de Fe\'riet hypergeometric  matrix function 
}

\author{
Ashish Verma\footnote{Corresponding author}
\\ 
Department of Mathematics\\ Prof. Rajendra Singh (Rajju Bhaiya)\\ Institute of Physical Sciences for Study and Research \\  V.B.S. Purvanchal University, Jaunpur  (U.P.)- 222003, India\\
vashish.lu@gmail.com
}

\maketitle
\begin{abstract}
 In this paper, we obtain  recursion formulas for the Kamp\'e de Fe\'riet hypergeometric  matrix function. We also give finite and  infinite summation  formulas for the Kamp\'e de Fe\'riet hypergeometric  matrix function. \\[12pt]
\end{abstract}

\section{Introduction}
Recursion formulas for the Appell function $F_2$ have been studied by Opps, Saad and Srivastava \cite{OP}, followed by Wang \cite{XW},  who presented the recursion relations for all Appell functions. The authors have carried out a systematic study of recursion formulas of multivariable hypergeometric functions including fourteen  three variable Lauricella functions, three Srivastava's triple hypergeometric functions  and four $k$--variable Lauricella functions \cite{VS1} and Exton's triple hypergeometric functions \cite{VS2}. The recursion formulas for the general triple hypergeometric function \cite{SK} were obtained in \cite{VS3}. These results were unified and generalized in \cite{VS1} for the three variable hypergeometric function. In \cite{VS4}, recursion formulas for the general  Kamp\'e de F\'eriet series and Srivastava and Daoust multivariable hypergeometric function were presented. In this paper, we begin a study of recursion formulas satisfied by one and two variable hypergeometric matrix functions. In particular, we present  recursion formulas for  the Gauss hypergeometric  and Appell matrix functions.
 
The matrix theory is being recently used in theory of orthogonal polynomials and special functions. Special matrix functions appear in the literature related to Statistics \cite{A}, Lie theory \cite{AT} and more recently in connection with the matrix  version of Laguerre, Hermite and Legendre differential equations and the corresponding polynomial families \cite{LR,LRE,LR1}.

Let $\mathbb{C}^{r\times r}$ be the vector space of $r$ square matrices with complex entries. For any matrix $A\in \mathbb{C}^{r\times r}$, its spectrum $\sigma(A)$ is the set of eigenvalues of $A$. If $f(z)$ and $g(z)$ are holomorphic functions of the complex variable $z$, which are defined in an open set $\Omega$ of the complex plane and $A\in \mathbb{C}^{r\times r}$ with $\sigma(A)\subset\Omega$, then from the properties of the matrix functional calculus \cite{ND}, we have  $f(A) g(A)= g(A) f(A)$. If $B\in\mathbb{C}^{r\times r}$ is a matrix for which $\sigma(B)\subset\Omega$, and if $AB=BA$  then $f(A) g(B)= g(B) f(A)$. A square matrix $A$ in $\mathbb{C}^{r\times r}$  is said to be positive stable if $\Re(\lambda)>0$ for all $\lambda\in\sigma(A)$. 

The reciprocal gamma function $\Gamma^{-1}(z)=1/\Gamma(z)$ is an entire function of the complex variable $z$. The image of $\Gamma^{-1}(z)$ acting on $A$, denoted by $\Gamma^{-1}(A)$, is a well defined matrix. If $A+nI$ is invertible for all integers $n\geq 0$, then the reciprocal gamma function \cite{LC1} is defined as 
$\Gamma^{-1}(A)= (A)_n \ \Gamma^{-1}(A+nI)$, where $(A)_n$ is the shifted factorial matrix function for $A\in\mathbb{C}^{r\times r}$ given by \cite{LC}
\begin{align*} 
	(A)_n=
\begin{cases}
I ,
& n=0,\\
 A(A+I) \cdots (A+(n-1)I) ,
&n\geq 1 .
\end{cases}
\end{align*}
$I$ being the $r$-square identity matrix.
If $A\in\mathbb{C}^{r\times r}$ is a positive stable matrix and $n\geq1$, then by \cite{LC1} we have $\Gamma(A) = \lim_{n \to\infty}(n-1)! (A)^{-1}_{n} n^A$.

The Gauss hypergeometric  matrix function \cite{LC} is defined by 
\begin{align}
_2F_1(A, B; C; x)=\sum_{n=0}^{\infty}\frac{(A)_{n}(B)_{n}(C)_{n}^{-1}}{n!}\, x^{n},\label{1eq1}
\end{align}
for matrices $A$, $B$ and $C$ in $\mathbb{C}^{r\times r}$ such that $C+kI$ is invertible for all $k\geq 0$ and $|x|\leq 1$.

The Appell  matrix functions are defined as 
\begin{align}
F_{1}(A, B, B'; C;x, y)&=\sum_{m, n=0}^{\infty}\frac{(A)_{m+n}(B)_{m}(B')_{n}(C)_{m+n}^{-1}}{m! n!}\, x^m y^n,\label{1eq2}
\\
F_{2}(A, B, B'; C, C';x, y)&=\sum_{m, n=0}^{\infty}\frac{(A)_{m+n}(B)_{m}(B')_{n}(C)_{m}^{-1}(C')_{n}^{-1}}{m! n!} \,x^m y^n,\label{1eq3}
\\
F_{3}(A, A', B, B'; C;x, y)&=\sum_{m, n=0}^{\infty}\frac{(A)_{m}(A')_{n}(B)_{m}(B')_{n}(C)_{m+n}^{-1}}{m! n!}\, x^m y^n,\label{1eq4}
\\
F_{4}(A, B; C, C';x, y)&=\sum_{m, n=0}^{\infty}\frac{(A)_{m+n}(B)_{m+n}(C)_{m}^{-1}(C')_{n}^{-1}}{m! n!} \,x^m y^n,\label{1eq5}
\end{align}
where $A$, $A'$, $B$, $B'$, $C$, $C'$  are positive stable matrices in $\mathbb{C}^{r\times r}$  such that $C+kI$  and $C'+kI$ are invertible for all integers $k\geq 0$. For regions of convergence of (\ref{1eq2})-(\ref{1eq5}), see \cite{QM,RD,RD3}.\\

The Kamp\'e de F\'eriet hypergeometric matrix function is defined by \cite{RD, RD3}
\begin{align}
&F^{m_1; n_1, n'_1}_{m_2; n_2, n'_2}\left(^{A: B, C}_{D: E, F}; x, y\right)\notag\\
&= \sum_{m, n \geq 0} \prod_{i=1}^{m_1}(A_i)_{m+n} \prod_{i=1}^{n_1}(B_i)_{m} \prod_{i=1}^{n'_1}(C_i)_{n} \prod_{i=1}^{m_2}(D_i)^{-1}_{m+n} \prod_{i=1}^{n_2}(E_i)^{-1}_{m} \prod_{i=1}^{n'_2}(F_i)^{-1}_{n} \frac{x^{m} y^{n}}{m! n!}\label{1eq6}
\end{align}
where $A$ abbreviates the sequence of matrices $A_1,\dots, A_{m_1},$ etc. and $A_i$, $B_i$, $C_i$, $D_i$, $E_i$, $F_i$ are positive stable matrices in $\mathbb{C}^{r\times r}$  such that $D_i+kI$, $E_i+kI$ and $F_i+kI$ are invertible for all integers $k\geq 0$.\\
Next, we recall the definition of derivative operator 
\begin{align*}
D_{y}f(y)=\lim_{h\to 0}\frac{f(y+h)-f(y)}{h},
\end{align*}
provided $f$ is differentiable at $y$. Also $D_{y}^{k}f{(y)}=D_{y}(D_{y}^{k-1}f(y))$,  $k=0,1,2,\dots.$

 Throughout the paper, $I$ denotes the identity matrix and $s$ denotes a non-negative integer. 
Following abbreviated notations are used. We, for example, write 
\begin{align}
&A+sI= A_1+sI, A_2+sI,\dots, A_{m_{1}}+sI\notag\\
&A^{i}=A_1,\dots,A_{i-1}, A_{i+1}, \dots, A_{m_1}\notag\\
&A^{i}+sI=A_1+sI,\dots,A_{i-1}+sI, A_{i+1}+sI, \dots, A_{m_1}+sI.\notag\end{align}
Also, we denote
\begin{align}
\qquad [A+kI]_{s}=\prod_{i=1}^{m_1}\,(A_i+kI)_{s}, \qquad [A+kI]^{-1}_{s}=\prod_{i=1}^{m_1}\,(A_i+kI)^{-1}_{s}\notag\\
\qquad
[A^j+kI]_{s}=\prod_{i=1, i\neq j}^{m_1}\,(A_i +kI)_{s}\qquad
[A^j+kI]^{-1}_{s}=\prod_{i=1, i\neq j}^{m_1}\,(A_i +kI)^{-1}_{s}.\notag
\end{align}

\section{Recursion formulas for the Kamp\'e de F\'eriet hypergeometric matrix function}
In this section, we obtain the recursion formulas for the Kamp\'e de F\'eriet hypergeometric matrix function. 
\begin{theorem}\label{rth12} Let $A_i+sI, i=1,\dots, m_1$ be  invertible  for all integers $s\geq0$. 
Then the following recursion formula hold  true for the  Kamp\'e de F\'eriet hypergeometric matrix function:
\begin{align}
&F^{m_1; n_1, n'_1}_{m_2; n_2, n'_2}\left(^{A^{i}, \,A_i+sI: B, C}_{\,\,\,\,\,\,\,\,D\,\,\,\,\,\,\,\,\,:\, E,\, F}; x, y\right)\notag\\&=F^{m_1; n_1, n'_1}_{m_2; n_2, n'_2}\left(^{A: B, C}_{D: E, F}; x, y\right)+ x[A^{i}][B]\Big[\sum_{k=1}^{s}F^{m_1; n_1, n'_1}_{m_2; n_2, n'_2}\left(^{A^{i}+I, \,A_i+kI: B+I, C}_{\,\,\,\,\,\,\,\,D+I\,\,\,\,\,:\, E+I,\, F}; x, y\right)\Big][D]^{-1}[E]^{-1}\notag\\
&+ y[A^{i}][C]\Big[\sum_{k=1}^{s}F^{m_1; n_1, n'_1}_{m_2; n_2, n'_2}\left(^{A^{i}+I, \,A_i+kI: B, C+I}_{\,\,\,\,\,\,\,\,D+I\,\,\,\,\,:\, E,\, F+I}; x, y\right)\Big][D]^{-1}[F]^{-1}.
\label{2eq1}
\end{align}
Furthermore, if $A_i-kI$ is invertible for integers  $k\leq s$, then
\begin{align}
&F^{m_1; n_1, n'_1}_{m_2; n_2, n'_2}\left(^{A^{i}, \,A_i-sI: B, C}_{\,\,\,\,\,\,\,\,D\,\,\,\,\,\,\,\,\,:\, E,\, F}; x, y\right)\notag\\&=F^{m_1; n_1, n'_1}_{m_2; n_2, n'_2}\left(^{A: B, C}_{D: E, F}; x, y\right)- x[A^{i}][B]\sum_{k=0}^{s-1}F^{m_1; n_1, n'_1}_{m_2; n_2, n'_2}\left(^{A^{i}+I, \,A_i-kI: B+I, C}_{\,\,\,\,\,\,\,\,D+I\,\,\,\,\,:\, E+I,\, F}; x, y\right)[D]^{-1}[E]^{-1}\notag\\
&- y[A^{i}][C]\sum_{k=0}^{s-1}F^{m_1; n_1, n'_1}_{m_2; n_2, n'_2}\left(^{A^{i}+I, \,A_i-kI: B, C+I}_{\,\,\,\,\,\,\,\,D+I\,\,\,\,\,:\, E,\, F+I}; x, y\right)[D]^{-1}[F]^{-1},
\label{2eq2}
\end{align}
where $A_i$, $B_i$, $C_i$, $D_i$, $E_i$, $F_i$ are positive stable matrices in $\mathbb{C}^{r\times r}$  such that $A_i A_j= A_j A_i$; $A_i B_j=B_j A_i$; $A_i C_{j}= C_{j} A_i$;  $B_i C_{j}= C_{j} B_i$;  $F_i E_j=E_j F_i$; $F_j D_{i}= D_{i} F_j$; $D_i E_j=E_j D_i$ and $D_i+kI$, $E_i+kI$ and $F_i+kI$ are invertible for all integers $k\geq 0$.
\end{theorem}
\begin{proof} 
From the definition of the  Kamp\'e de F\'eriet hypergeometric matrix function (\ref{1eq6}) and  the relation
\begin{align}
(A_i+I)_{m+n}= A_{i}^{-1}(A_i)_{m+n}\left(A_i+{mI}+{nI}\right) ,\label{2eqp1}
\end{align}
we get the following contiguous matrix  relation:
\begin{align}
&F^{m_1; n_1, n'_1}_{m_2; n_2, n'_2}\left(^{A^{i},\, A_i+I: B, C}_{\,\,\,\,\,\,\,\,D\,\,\,\,\,\,\,\,\,:\, E,\, F}; x, y\right)\notag\\&=F^{m_1; n_1, n'_1}_{m_2; n_2, n'_2}\left(^{A: B, C}_{D: E, F}; x, y\right)+ x[A^{i}][B]\Big[ F^{m_1; n_1, n'_1}_{m_2; n_2, n'_2}\left(^{A+I: B+I, C}_{D+I:\, E+I,\, F}; x, y\right)\Big][D]^{-1}[E]^{-1}\notag\\
&+ y[A^{i}][C]\Big[ F^{m_1; n_1, n'_1}_{m_2; n_2, n'_2}\left(^{A+I: B, C+I}_{D+I : E, F+I}; x, y\right)\Big][D]^{-1}[F]^{-1}.
\label{2eqp2}
\end{align}
Replacing $A_i$ with $A_i+I$ in (\ref{2eqp2}) and using (\ref{2eqp2}), we have the following contiguous  matrix relation:
\begin{align}
&F^{m_1; n_1, n'_1}_{m_2; n_2, n'_2}\left(^{A^{i},\, A_i+2I: B, C}_{\,\,\,\,\,\,\,\,D\,\,\,\,\,\,\,\,\,:\, E,\, F}; x, y\right)\notag\\&=F^{m_1; n_1, n'_1}_{m_2; n_2, n'_2}\left(^{A: B, C}_{D: E, F}; x, y\right)+ x[A^{i}][B]\Big[\sum_{k=1}^{2}F^{m_1; n_1, n'_1}_{m_2; n_2, n'_2}\left(^{A^{i}+I, \,A_i+kI: B+I, C}_{\,\,\,\,\,\,\,\,D+I\,\,\,\,\,:\, E+I,\, F}; x, y\right)\Big][D]^{-1}[E]^{-1}\notag\\
&+ y[A^{i}][C]\Big[\sum_{k=1}^{2}F^{m_1; n_1, n'_1}_{m_2; n_2, n'_2}\left(^{A^{i}+I,\, A_i+kI: B, C+I}_{\,\,\,\,\,\,\,\,D+I\,\,\,\,\,:\, E,\, F+I}; x, y\right)\Big][D]^{-1}[F]^{-1}.
\label{2eqp3}
\end{align}
Iterating this process $s$ times, we get (\ref{2eq1}).  For the proof of (\ref{2eq2}) replace the matrix $A_i$ with $A_i-I$ in (\ref{2eqp2}). As $A_i-I$ is invertible, we have
\begin{align}
&F^{m_1; n_1, n'_1}_{m_2; n_2, n'_2}\left(^{A^{i}, \,A_i-I: B, C}_{\,\,\,\,\,\,\,\,D\,\,\,\,\,\,\,\,\,:\, E,\, F}; x, y\right)\notag\\&=F^{m_1; n_1, n'_1}_{m_2; n_2, n'_2}\left(^{A: B, C}_{D: E, F}; x, y\right)- x[A^{i}][B]\Big[ F^{m_1; n_1, n'_1}_{m_2; n_2, n'_2}\left(^{A^{i}+I,\,A_i\,:  B+I, C}_{\,\,\,\,D+I\,\,\,\, :\, E+I,\, F}; x, y\right)\Big][D]^{-1}[E]^{-1}\notag\\
&- y[A^{i}][C]\Big[F^{m_1; n_1, n'_1}_{m_2; n_2, n'_2}\left(^{A^{i}+I,\,A_i\,: B, C+I}_{ \,\,\,\,D+I\,\,\,:\, E,\, F+I}; x, y\right)\Big][D]^{-1}[F]^{-1}.\label{2eqp4}
\end{align}
Iteratively, we get (\ref{2eq2}).
\end{proof}

Using contiguous matrix relations (\ref{2eqp2}) and (\ref{2eqp4}), we get following forms of the recursion formulas for the  Kamp\'e de F\'eriet hypergeometric matrix function.
\begin{theorem}\label{rth21}
Let $A_i+sI, i=1,\dots, m_1$ be  invertible  for all integers $s\geq0$. Then the following recursion formula holds  true for the  Kamp\'e de F\'eriet hypergeometric matrix function:
\begin{align}
&F^{m_1; n_1, n'_1}_{m_2; n_2, n'_2}\left(^{A^{i}, \,A_i+sI: B, C}_{\,\,\,\,\,\,\,\,D\,\,\,\,\,\,\,\,\,:\, E,\, F}; x, y\right)\notag\\&=\sum_{k_1+k_2\leq s}^{}{s\choose k_1, k_2}{[A^{i}]_{k_1+k_2}[B]_{k_1}}[C]_{k_2}\, \notag\\&\times x^{k_1}y^{k_2} F^{m_1; n_1, n'_1}_{m_2; n_2, n'_2}\left(^{ A+(k_1+k_2)I: B+k_1I, C+k_2I}_{D+(k_1+k_2)I:\, E+k_1 I, F+k_2I}; x, y\right){[D]^{-1}_{k_1+k_2}[E]^{-1}_{k_1}[F]^{-1}_{k_2}}.
\label{2eq3}
\end{align}
Furthermore, if $A_i-kI$ is invertible for integers $k\leq s$, then
\begin{align}
&F^{m_1; n_1, n'_1}_{m_2; n_2, n'_2}\left(^{A^{i}, \,A_i-sI: B, C}_{\,\,\,\,\,\,\,\,D\,\,\,\,\,\,\,\,\,:\, E,\, F}; x, y\right)\notag\\&=\sum_{k_1+k_2\leq s}^{}{s\choose k_1, k_2}{[A^{i}]_{k_1+k_2}[B]_{k_1}}[C]_{k_2}\,\notag\\&\times  (-x)^{k_1} (-y)^{k_2} F^{m_1; n_1, n'_1}_{m_2; n_2, n'_2}\left(^{ A^{i}+(k_1+k_2)I, \,A_i: B+k_1I, C+k_2I}_{D+(k_1+k_2)I:\, E+k_1I,\, F+k_2I}; x, y\right){[D]^{-1}_{k_1+k_2}[E]^{-1}_{k_1}[F]^{-1}_{k_2}},\label{2eq4}
\end{align}
where $A_i$, $B_i$, $C_i$, $D_i$, $E_i$, $F_i$ are positive stable matrices in $\mathbb{C}^{r\times r}$  such that $A_i A_j= A_j A_i$; $A_i B_j=B_j A_i$; $A_i C_{j}= C_{j} A_i$;  $B_i C_{j}= C_{j} B_i$;  $F_i E_j=E_j F_i$; $F_j D_{i}= D_{i} F_j$; $D_i E_j=E_j D_i$ and $D_i+kI$, $E_i+kI$ and $F_i+kI$ are invertible for all integers $k\geq 0$.
\end{theorem}
\begin{proof}  
We prove (\ref{2eq3}) by applying mathematical induction on $s$. For $s=1$, the result (\ref{2eq3}) is true by (\ref{2eqp2}). Assuming (\ref{2eq3}) is true for $s=t$, that is, 
\begin{align}
&F^{m_1; n_1, n'_1}_{m_2; n_2, n'_2}\left(^{A^{i}, \,A_i+tI: B, C}_{\,\,\,\,\,\,\,\,D\,\,\,\,\,\,\,\,\,:\, E,\, F}; x, y\right)\notag\\&=\sum_{k_1+k_2\leq t}^{}{t\choose k_1, k_2}{[A^{i}]_{k_1+k_2}[B]_{k_1}}[C]_{k_2}\, \notag\\&\times x^{k_1}y^{k_2} F^{m_1; n_1, n'_1}_{m_2; n_2, n'_2}\left(^{ A+(k_1+k_2)I: B+k_1I, C+k_2I}_{D+(k_1+k_2)I:\, E+k_1 I, F+k_2I}; x, y\right){[D]^{-1}_{k_1+k_2}[E]^{-1}_{k_1}[F]^{-1}_{k_2}}.
\label{2eqpp11}
\end{align}
Replacing $A_i$ with $A_i+I$ in (\ref{2eqpp11}) and using the contiguous matrix  relation 
(\ref{2eqp2}) and simplifying, we get 
%\begin{align}
% &F^{m_1; n_1, n'_1}_{m_2; n_2, n'_2}\left(^{A^{i}, \,A_i+tI+I: B, C}_{\,\,\,\,\,\,\,\,D\,\,\,\,\,\,\,\,\,:\, E,\, F}; x, y\right)\notag\\
% &=\sum_{k_1+k_2\leq t}^{}{t\choose k_1, k_2}{[A^{i}]_{k_1+k_2}[B]_{k_1}}[C]_{k_2}\, x^{k_1} y^{k_2}\,\notag\\
% &\quad\times \{F^{m_1; n_1, n'_1}_{m_2; n_2, n'_2}\left(^{ A+(k_1+k_2)I: B+k_1I, C+k_2I}_{D+(k_1+k_2)I:\, E+k_1 I, F+k_2I}; x, y\right)+x[A^{i}+(k_1+k_2)I][B+k_1I]\notag\\
% &\quad F^{m_1; n_1, n'_1}_{m_2; n_2, n'_2}\left(^{ A+(k_1+k_2)I+I: B+k_1 I+I, C+k_2I}_{D+(k_1+k_2)I+I:\, E+k_1 I+I, F+k_2I}; x, y\right)[D+(k_1+k_2)I]^{-1}[E+k_1I]^{-1}+y[A^{i}+(k_1+k_2)I]\notag\\&\quad\times[C+k_1I]
%  F^{m_1; n_1, n'_1}_{m_2; n_2, n'_2}\left(^{ A+(k_1+k_2)I+I: B+k_1 I, C+k_2I+I}_{D+(k_1+k_2)I+I:\, E+k_1 I, F+k_2I+I}; x, y\right)[D+(k_1+k_2)I]^{-1}[F+k_2I]^{-1}\}\notag\\&\quad{[D]^{-1}_{k_1+k_2}[E]^{-1}_{k_1}[F]^{-1}_{k_2}}.
% \label{2eqpp2}
% \end{align}
% Simplifying, (\ref{2eqpp2}) takes the form
 \begin{align}
 &F^{m_1; n_1, n'_1}_{m_2; n_2, n'_2}\left(^{A^{i},\, A_i+tI+I: B, C}_{\,\,\,\,\,\,\,\,D\,\,\,\,\,\,\,\,\,:\, E,\, F}; x, y\right)\notag\\
 &=\sum_{k_1+k_2\leq t}^{}{t\choose k_1, k_2}{[A^{i}]_{k_1+k_2}[B]_{k_1}}[C]_{k_2}\, x^{k_1} y^{k_2}\,\notag\\
 &\quad F^{m_1; n_1, n'_1}_{m_2; n_2, n'_2}\left(^{ A+(k_1+k_2)I: B+k_1I, C+k_2I}_{D+(k_1+k_2)I:\, E+k_1 I, F+k_2I}; x, y\right){[D]^{-1}_{k_1+k_2}[E]^{-1}_{k_1}[F]^{-1}_{k_2}}\notag\\
 &\quad +\sum_{k_1+k_2\leq t+1}^{}{t\choose k_1-1, k_2}{[A^{i}]_{k_1+k_2}[B]_{k_1}}[C]_{k_2}\, x^{k_1} y^{k_2}\notag\\&\quad\times F^{m_1; n_1, n'_1}_{m_2; n_2, n'_2}\left(^{ A+(k_1+k_2)I: B+k_1I, C+k_2I}_{D+(k_1+k_2)I:\, E+k_1 I, F+k_2I}; x, y\right){[D]^{-1}_{k_1+k_2}[E]^{-1}_{k_1}[F]^{-1}_{k_2}}\notag\\
 &\quad +\sum_{k_1+k_2\leq t+1}^{}{t\choose k_1, k_2-1}{[A^{i}]_{k_1+k_2}[B]_{k_1}}[C]_{k_2}\, x^{k_1} y^{k_2}\notag\\&\quad\times  F^{m_1; n_1, n'_1}_{m_2; n_2, n'_2}\left(^{ A+(k_1+k_2)I: B+k_1I, C+k_2I}_{D+(k_1+k_2)I:\, E+k_1 I, F+k_2I}; x, y\right){[D]^{-1}_{k_1+k_2}[E]^{-1}_{k_1}[F]^{-1}_{k_2}}.
\end{align}

Apply the known relation ${n\choose k}+{n\choose k-1}={n+1\choose k}$ and ${n\choose k}=0$\, ($k>n$ or $k<0$), the above identity can be reduced to the following result:
\begin{align}
&F^{m_1; n_1, n'_1}_{m_2; n_2, n'_2}\left(^{A^{i},\, A_i+tI+I: B, C}_{\,\,\,\,\,\,\,\,D\,\,\,\,\,\,\,\,\,:\, E,\, F}; x, y\right)\notag\\&=\sum_{k_1+k_2\leq t+1}^{}{t+1\choose k_1, k_2}{[A^{i}]_{k_1+k_2}[B]_{k_1}}[C]_{k_2}\, \notag\\&\times x^{k_1}y^{k_2} F^{m_1; n_1, n'_1}_{m_2; n_2, n'_2}\left(^{ A+(k_1+k_2)I: B+k_1I, C+k_2I}_{D+(k_1+k_2)I:\, E+k_1 I, F+k_2I}; x, y\right){[D]^{-1}_{k_1+k_2}[E]^{-1}_{k_1}[F]^{-1}_{k_2}}.
\label{2eqpp1}
\end{align}
This establishes (\ref{2eq3}) for $s= t+1$.  Hence result (\ref{2eq3}) is true for all values of $s$. The second recursion formula (\ref{2eq4}) is proved in a similar manner.
\end{proof} 

 Now, we present the recursion formulas for matrix $B_i$, $C_i$ of the Kamp\'e de F\'eriet hypergeometric matrix function. We omit the proofs of the given below theorems.
\begin{theorem}\label{rth12}Let $B_i+sI, i=1,\dots, n_1$ be  invertible  for all integers $s\geq0$.    Then the following recursion formula holds  true for the  Kamp\'e de F\'eriet hypergeometric matrix function:
\begin{align}
&F^{m_1; n_1, n'_1}_{m_2; n_2, n'_2}\left(^{A:B^{i},\, B_i+sI, C}_{D:\, E,\, F}; x, y\right)\notag\\&=F^{m_1; n_1, n'_1}_{m_2; n_2, n'_2}\left(^{A: B, C}_{D: E, F}; x, y\right)\notag\\&\quad + x[A][B^{i}]\sum_{k=1}^{s}F^{m_1; n_1, n'_1}_{m_2; n_2, n'_2}\left(^{A+I: B^{i}+I, \,B_i+kI, C}_{D+I :\, E+I,\, F}; x, y\right)[D]^{-1}[E]^{-1}.
\label{2eq5}
\end{align}
Furthermore, if $B_i-kI$ is invertible for integers $k\leq s$, then
\begin{align}
&F^{m_1; n_1, n'_1}_{m_2; n_2, n'_2}\left(^{A:B^{i},\, B_i-sI, C}_{D:\, E,\, F}; x, y\right)\notag\\&=F^{m_1; n_1, n'_1}_{m_2; n_2, n'_2}\left(^{A: B, C}_{D: E, F}; x, y\right)\notag\\&\quad - x[A][B^{i}]\sum_{k=0}^{s-1}F^{m_1; n_1, n'_1}_{m_2; n_2, n'_2}\left(^{A+I: B^{i}+I, \,B_i-kI, C}_{D+I :\, E+I,\, F}; x, y\right)[D]^{-1}[E]^{-1},
\label{2eq6}
\end{align}
where $A_i$, $B_i$, $C_i$, $D_i$, $E_i$, $F_i$ are positive stable matrices in $\mathbb{C}^{r\times r}$  such that $A_i B_j=B_j A_i$; $B_i B_j= B_j B_i$;  $ F_j D_{i}= D_{i} F_j$; $D_i E_j=E_j D_i$;   $F_i E_j=E_j F_i$ and $D_i+kI$, $E_i+kI$ and $F_i+kI$ are invertible for all integers $k\geq 0$.
\end{theorem}
\begin{theorem}\label{rth13}
  Let $B_i+sI, i=1,\dots, n_1$ be  invertible  for all  integers $s\geq0$. 
 Then the following recursion formula holds  true for the  Kamp\'e de F\'eriet hypergeometric matrix function:
\begin{align}
F^{m_1; n_1, n'_1}_{m_2; n_2, n'_2}\left(^{A: B^{i}, \,B_i+sI, C}_{D:\,\,\,\,\,\,\, E,\,\,\, F}; x, y\right)&=\sum_{k=0}^{s}{s\choose k}{[A]_{k}[B^{i}]_{k}}\, x^{k}\notag\\&\times F^{m_1; n_1, n'_1}_{m_2; n_2, n'_2}\left(^{ A+kI: B+kI, C}_{D+kI:\, E+kI,\, F}; x, y\right){[D]^{-1}_{k}[E]^{-1}_{k}},
\label{2eq7}
\end{align}
Furthermore, if $B_i-kI$ is invertible for integers $k\leq s$, then
\begin{align}
F^{m_1; n_1, n'_1}_{m_2; n_2, n'_2}\left(^{A : B^{i}, \,B_i-sI, C}_{\,D:\,\,\,\, E,\,\,\, F}; x, y\right)&=\sum_{k=0}^{s}{s\choose k}{[A]_{k}[B^{i}]_{k}}\, (-x)^{k}\notag\\&\times F^{m_1; n_1, n'_1}_{m_2; n_2, n'_2}\left(^{ A+kI : B^{i}+kI,\, B_i, C}_{D+kI:\, E+kI,\, F}; x, y\right){[D]^{-1}_{k}[E]^{-1}_{k}}.\label{2eq8}
\end{align}
where $A_i$, $B_i$, $C_i$, $D_i$, $E_i$, $F_i$ are positive stable matrices in $\mathbb{C}^{r\times r}$  such that $A_i B_j=B_j A_i$; $B_i B_j= B_j B_i$; $F_j D_{i}= D_{i} F_j$; $D_i E_j=E_j D_i$;   $F_i E_j=E_j F_i$ and $D_i+kI$, $E_i+kI$ and $F_i+kI$ are invertible for all integers $k\geq 0$.
\end{theorem}
The recursion formulas for $F^{m_1; n_1, n'_1}_{m_2; n_2, n'_2}\left(^{A:  B, C^{i}, \,C_i\pm sI}_{D:\,\,\, E,\,\,\, F}; x, y\right)$  are obtained by replacing $B \leftrightarrow C$, $E \leftrightarrow F$ and  $x\leftrightarrow y$ in Theorems~\ref{rth12} --~\ref{rth13}.\\

 Next, we present the recursion formulas for matrix  $D_i$ of the Kamp\'e de F\'eriet hypergeometric matrix function. 
\begin{theorem}
Let $D_i-sI, i=1,\dots, m_2$ be  invertible  for all integers $s\geq0$.  Then the following recursion formula holds  true for the  Kamp\'e de F\'eriet hypergeometric matrix function:
\begin{align}
&F^{m_1; n_1, n'_1}_{m_2; n_2, n'_2}\left(^{\,\,\,\,\,\,\,\,\,\,\,\,A\,\,\,\,\,\,\,\,\,\,\,\,: B, C}_{D^{i},\, D_i-sI\,\,:\, E,  F}; x, y\right)\notag\\&=F^{m_1; n_1, n'_1}_{m_2; n_2, n'_2}\left(^{A: B, C}_{D: E, F}; x, y\right)\notag\\&+ x[A][B]\Big[\sum_{k=1}^{s}F^{m_1; n_1, n'_1}_{m_2; n_2, n'_2}\left(^{\,\,\,\,\,\,\,\,\,\,\,\,\,\,A+I\,\,\,\,\,\,\,\,\,\,\,\,\,: B+I, C}_{D^{i}+I, \,D_i+(2-k)I:\, E+I,\, F}; x, y\right)(D_i-kI)^{-1}(D_i-(k-1)I)^{-1}\Big][D^{i}]^{-1}[E]^{-1}\notag\\
&+ y[A][C]\Big[\sum_{k=1}^{s}F^{m_1; n_1, n'_1}_{m_2; n_2, n'_2}\left(^{\,\,\,\,\,\,\,\,\,\,\,\,\,\,\,\,\,A+I\,\,\,\,\,\,\,\,\,\,: B, C+I}_{D^{i}+I, \,D_i+(2-k)I:\, E,\, F+I}; x, y\right)(D_i-kI)^{-1}(D_i-(k-1)I)^{-1}\Big][D^{i}]^{-1}[F]^{-1},\label{2eq9}
\end{align}
where $A_i$, $B_i$, $C_i$, $D_i$, $E_i$, $F_i$ are positive stable matrices in $\mathbb{C}^{r\times r}$  such that $A_i B_j=B_j A_i$; $A_i C_{j}= C_{j} A_i$;   $B_i C_{j}= C_{j} B_i$; $F_i E_j=E_j F_i$; $D_i D_j= D_j D_i$; $F_j D_{i}= D_{i} F_j$; $D_i E_j=E_j D_i$ and $D_i+kI$, $E_i+kI$ and $F_i+kI$ are invertible for all integers $k\geq 0$.
\end{theorem}
\begin{proof} Applying the definition of the Kamp\'e de F\'eriet hypergeometric matrix function and the relation
\begin{align*}
{(D_i-I)^{-1}_{m+n}}={(D_i)^{-1}_{m+n}}\left[1+{m}{(D_i-I)^{-1}}+{n}{(D_i-I)^{-1}}\right],
\end{align*} 
we obtain the following contiguous  matrix relation:
\begin{align}
&F^{m_1; n_1, n'_1}_{m_2; n_2, n'_2}\left(^{\,\,\,\,\,\,\,\,\,\,\,\,A\,\,\,\,\,\,\,\,\,\,\,\,: B, C}_{D^{i}, \,D_i-I\,\,:\, E,  F}; x, y\right)\notag\\&=F^{m_1; n_1, n'_1}_{m_2; n_2, n'_2}\left(^{A: B, C}_{D: E, F}; x, y\right)\notag\\&+ x[A][B]\Big[F^{m_1; n_1, n'_1}_{m_2; n_2, n'_2}\left(^{\,A+I\,: B+I, C}_{D+I:\, E+I,\, F}; x, y\right)(D_i-I)^{-1}\Big][D]^{-1}[E]^{-1}\notag\\
&+ y[A][C]\Big[F^{m_1; n_1, n'_1}_{m_2; n_2, n'_2}\left(^{\,A+I\,: B, C+I}_{D+I:\, E,\, F+I}; x, y\right)(D_i-I)^{-1}\Big][D]^{-1}[F]^{-1}.
\end{align}
Using this contiguous matrix relation to the  Kamp\'e de F\'eriet hypergeometric matrix function with the matrix $D_i-sI$ for $s$ times, we get (\ref{2eq9}). 
\end{proof} 

Next, we will present recursion formulas for the  Kamp\'e de F\'eriet hypergeometric matrix function $E_{i}$, $F_{i}$. We omit the proof of the given below theorem.
\begin{theorem}\label{sa1}
Let $E_i-sI, i=1,\dots, n_2$ be  invertible  for all integers $s\geq0$.  Then the following recursion formula holds  true for the  Kamp\'e de F\'eriet hypergeometric matrix function:
\begin{align}
&F^{m_1; n_1, n'_1}_{m_2; n_2, n'_2}\left(^{\,A:\,\,\,\,\, B,\,\,\,\, C}_{D:\,E^{i},\, E_i-sI,  F}; x, y\right)\notag\\&=F^{m_1; n_1, n'_1}_{m_2; n_2, n'_2}\left(^{A: B, C}_{D: E, F}; x, y\right)\notag\\&+ x[A][B]\Big[\sum_{k=1}^{s}F^{m_1; n_1, n'_1}_{m_2; n_2, n'_2}\left(^{A+I\,: \,\,B+I,\,\, C}_{D+I:\,E^{i}+I, \, E_i+(2-k)I,\, F}; x, y\right)(E_i-kI)^{-1}(E_i-(k-1)I)^{-1}\Big][E^{i}]^{-1}[D]^{-1},\label{2eq10}
\end{align}
where $A_i$, $B_i$, $C_i$, $D_i$, $E_i$, $F_i$ are positive stable matrices in $\mathbb{C}^{r\times r}$  such that $A_i B_j=B_j A_i$; $E_i E_j= E_j E_i$; $F_i E_j=E_j F_i$; $F_j D_{j}= D_{j} F_i$; $D_j E_i=E_i D_j$ and $D_i+kI$, $E_i+kI$ and $F_i+kI$ are invertible for all integers $k\geq 0$.
\end{theorem}

The recursion formulas for $F^{m_1; n_1, n'_1}_{m_2; n_2, n'_2}\left(^{\,A: B,\,\,\,\, C\,\,\,\,\,}_{D:E, F^{i},\, F_i-sI }; x, y\right)$  are obtained by replacing $B \leftrightarrow C$, $E \leftrightarrow F$ and  $x\leftrightarrow y$ in Theorem~\ref{sa1}.\\
%\begin{theorem}Let $F_i-sI, i=1,\dots, n'_2$ be  invertible  for all integers $s\geq0$.  Then the following recursion formula holds  true for the  Kamp\'e de F\'eriet hypergeometric matrix function:
%\begin{align}
%&F^{m_1; n_1, n'_1}_{m_2; n_2, n'_2}\left(^{\,A: B,\,\,\,\, C\,\,\,\,\,}_{D:E, F^{i},\, F_i-sI }; x, y\right)\notag\\&=F^{m_1; n_1, n'_1}_{m_2; n_2, n'_2}\left(^{A: B, C}_{D: E, F}; x, y\right)\notag\\&+ y[A][C]\Big[\sum_{k=1}^{s}F^{m_1; n_1, n'_1}_{m_2; n_2, n'_2}\left(^{\,\,\,\,A+I\,\,\,\,\,\,\,: B, C+I}_{D+I:\, E,\,F^{i},\, F_i+(2-k)I}; x, y\right)(F_i-kI)^{-1}(F_i-(k-1)I)^{-1}\Big][D]^{-1},\label{2eq11}
%\end{align}
%where $A_i$, $B_i$, $C_i$, $D_i$, $E_i$, $F_i$ are positive stable matrices in $\mathbb{C}^{r\times r}$  such that $A_i C_j=C_j A_i$ ; $B_i C_j=C_j B_i$;  $F_i F_j= F_j F_i$;  $F_j D_{i}= D_{i} F_j$; $D_i E_j=E_j D_i$ and $D_i+kI$, $E_i+kI$ and $F_i+kI$ are invertible for all integers $k\geq 0$.
%\end{theorem}
\section{Finite matrix summation  formulas for the  Kamp\'e de F\'eriet hypergeometric matrix function by derivative operator}
In this section,  we  obtain the finite matrix sumation  formulas  for the  Kamp\'e de F\'eriet hypergeometric matrix function by derivative operator. The $p$-th derivative on $y$ of  Kamp\'e de F\'eriet hypergeometric matrix function is obtained as follows:
\begin{align}
& D_{y}^{p}\{F^{m_1; n_1, n'_1}_{m_2; n_2, n'_2}\left(^{A: B, C}_{D: E, F}; x, y\right)\}\notag\\
&=[A]_{p}[C]_{p}\,\,F^{m_1; n_1, n'_1}_{m_2; n_2, n'_2}\left(^{A+pI: B, C+pI}_{D+pI: E, F+pI}; x, y\right)[D]^{-1}_{p}[F]_{p}^{-1},\label{s1}
\end{align}
 where $A_i$, $B_i$, $C_i$, $D_i$, $E_i$, $F_i$ are positive stable matrices in $\mathbb{C}^{r\times r}$  such that  $A_iC_j=C_j A_i$;  $B_iC_j= C_j B_i$; $D_i E_j= E_j D_i$; $D_iF_j= F_j D_i$ and  $D_i+kI$, $E_i+kI$ and $F_i+kI$ are invertible for all integers $k\geq 0$.

By using the generalized Leibnitz formula
\begin{align*}
D_{y}^{p} \ \big(f(y) g(y)\big)=\sum_{k=0}^{p}{p\choose k} \ D_{y}^{p-k}f(y) \ D_{y}^{k}g(y)
\end{align*}
and  (\ref{s1}), we derive the following finite matrix summation formulas of   Kamp\'e de F\'eriet hypergeometric matrix function.
\begin{theorem}\label{t11} Let $A_i$, $B_i$, $C_i$, $D_i$, $E_i$, $F_i$  be positive stable matrices in $\mathbb{C}^{r\times r}$ such that   $A_iC_j=C_j A_i$;  $B_iC_j= C_j B_i$; $C_i C_j= C_j C_i$; $D_i E_j= E_j D_i$; $D_iF_j= F_j D_i$ and $D_i+kI$, $E_i+kI$ and $F_i+kI$ are invertible for all integers $k\geq 0$. Then the following finite matrix  summation formulas  hold for   Kamp\'e de F\'eriet hypergeometric matrix function:
\begin{align}
&\sum_{k=0}^{p}{p\choose k}{[A]_{k}[C^{i}]_{k}}\,y^{k}F^{m_1; n_1, n'_1}_{m_2; n_2, n'_2}\left(^{A+kI: B, C+kI}_{D+kI: E, F+kI}; x, y\right)[D]_{k}^{-1}[F]^{-1}_{k}\notag\\
&=F^{m_1; n_1, n'_1}_{m_2; n_2, n'_2}\left(^{A: B, C^{i},\, C_{i}+pI}_{D: E, F}; x, y\right).\label{1s21}
\end{align}
\end{theorem}
\begin{proof} From  definition of Kamp\'e de F\'eriet hypergeometric matrix function and the generalized Leibnitz formula for differentiation of a product of two functions, we have 
\begin{align*}
&D_{y}^{p}\{y^{C_{i}+pI-I}F^{m_1; n_1, n'_1}_{m_2; n_2, n'_2}\left(^{A: B, C}_{D: E, F}; x, y\right)\}\\
&= \sum_{k=0}^{p}{p\choose k}D_{y}^{p-k}\{y^{C_{i}+pI-I}\}D_{y}^{k}\{F^{m_1; n_1, n'_1}_{m_2; n_2, n'_2}\left(^{A: B, C}_{D: E, F}; x, y\right)\}\\
&=\,(C_{i})_{p}\, y^{C_{i}-I}\sum_{k=0}^{p}{p\choose k}{[A]_{k}[C^{i}]_{k}}\,\\
&\qquad\times y^{k}  F^{m_1; n_1, n'_1}_{m_2; n_2, n'_2}\left(^{A+kI: B, C+kI}_{D+kI: E, F+kI}; x, y\right)[D]^{-1}_{k}[F]^{-1}_{k},
\end{align*}
where, using  (\ref{s1}) and some simplification in the second equality. Next, we combine $y^{C_{i}+(p-1)I}$ with the variable $y$ in the Kamp\'e de F\'eriet hypergeometric matrix function and apply the derivative operator $p$-times on $y$ to get the following result:
\begin{align*}
&D_{y}^{p}\{y^{C_{i}+(p-1)I} F^{m_1; n_1, n'_1}_{m_2; n_2, n'_2}\left(^{A: B, C}_{D: E, F}; x, y\right)\}\\
&=\sum_{m, n=0}^{\infty}(C_{i}+nI)_{p}\prod_{i=1}^{m_1}(A_i)_{m+n} \prod_{i=1}^{n_1}(B_i)_{m} \prod_{i=1}^{n'_1}(C_i)_{n} \prod_{i=1}^{m_2}(D_i)^{-1}_{m+n} \prod_{i=1}^{n_2}(E_i)^{-1}_{m} \prod_{i=1}^{n'_2}(F_i)^{-1}_{n}\notag\\
&\times \frac{\, x^{m}y^{C_{i}+(n-1)I}}{m!n!}\\
&=(C_{i})_{p}\,y^{C_{i}-I}\,F^{m_1; n_1, n'_1}_{m_2; n_2, n'_2}\left(^{A: B, \,C_{i}+pI,\, C^{i}}_{D:\, E, \,F}; x, y\right).
\end{align*}
Equating the above two relations leads to (\ref{1s21}).
\end{proof}
%\begin{corollary}  Let $AB=BA$, $B'B=BB'$, Then following  finite summation formulas for the Appell matrix functions holds true:
%\begin{align}
%&\sum_{k=0}^{r}{r\choose k}(A)_{k}\, y^{k}\, F_{1}(A+kI, B, B'+kI; C+kI; x, y)(C)_{k}^{-1}\notag\\&=  F_{1}(A, B, B'+rI; C; x, y);\\
%&\sum_{k=0}^{r}{r\choose k}(A)_{k}\, y^{k} \,F_{2}(A+kI, B, B'+kI; C,  C'+kI; x, y)(C')_{k}^{-1}\notag\\&=  F_{2}(A, B, B'+rI; C, C'; x, y).
%\end{align}
%\end{corollary}
\begin{theorem}Let $A_i$, $B_i$, $C_i$, $D_i$, $E_i$, $F_i$  be positive stable matrices in $\mathbb{C}^{r\times r}$ such that   $A_iC_j=C_j A_i$;  $B_iC_j= C_j B_i$; $D_i E_j= E_j D_i$; $D_iF_j= F_j D_i$; $F_i F_j= F_j F_i$  and $D_i+kI$, $E_i+kI$ and $F_i+kI$ are invertible for all integers $k\geq 0$. Then the following   finite matrix summation formulas of   Kamp\'e de F\'eriet hypergeometric matrix function hold true:
\begin{align}
&\sum_{k=0}^{p}{p\choose k}{[A]_{k}[C]_{k}}\,y^{k}\,F^{m_1; n_1, n'_1}_{m_2; n_2, n'_2}\left(^{A+kI: B, C+kI}_{D+kI: E, F+kI}; x, y\right){(F_{i}-pI)^{-1}_{k}[D]^{-1}_{k}[F]^{-1}_{k}}\notag\\
&=\,F^{m_1; n_1, n'_1}_{m_2; n_2, n'_2}\left(^{A: B, C}_{D: E, F_{i}-pI,\, F^{i}}; x, y\right).\label{s3}
\end{align}
where $F_i+(k-p)I$  is an invertible matrix for $0\leq k \leq p$ and $i=1,\dots,n'_{2}.$
\end{theorem}
\begin{proof} Applying  the derivative operator on $y^{F_{i}-I}$$F^{m_1; n_1, n'_1}_{m_2; n_2, n'_2}\left(^{A: B, C}_{D: E, F}; x, y\right)$,  $p$-times, gives the formula in this theorem as explained in the proof of Theorem \ref{t11}. We omit the details.\end{proof}

%By specializing the parameters in this theorem, we get several finite summation formulas
%of Appell’s  matrix functions as follows.
%
%\begin{corollary}Let $AB'=B'A$,  $AB=BA$, $BB'=B'B$.   Then the following  finite summation formulas for the Appell matrix functions holds true:
%\begin{align}
%&\sum_{k=0}^{r}{r\choose k}(A)_{k} (B')_{k} \,y^{k} F_{2}(A+kI, B, B'+kI; C,  C'+kI; x, y)(C')^{-1}_{k}(C'-rI)^{-1}_{k}\notag\\&=  F_{2}(A, B, B'; C, C'-rI; x, y)\\
%&\sum_{k=0}^{r}{r\choose k}(A)_{k} (B)_{k}\, y^{k} F_{4}(A+kI, B+kI; C, C'+kI; x, y)(C')^{-1}_{k}(C'-rI)^{-1}_{k}\notag\\&=  F_{4}(A,   B; C, C'-rI; x, y).
%\end{align}where $C'+(k-r)I$  is an invertible matrix for $0\leq k \leq r$.
%\end{corollary}
Applying derivative operator and some simple transformations,  we can get the finite matrix
summation formulas of Kamp\'e de F\'eriet hypergeometric matrix function as follows.

\begin{theorem}Let $A_i$, $B_i$, $C_i$, $D_i$, $E_i$, $F_i$  be positive stable matrices in $\mathbb{C}^{r\times r}$ such that   $A_iC_j=C_j A_i$;  $B_iC_j= C_j B_i$; $D_i E_j= E_j D_i$; $D_iF_j= F_j D_i$  and $D_i+kI$, $E_i+kI$ and $F_i+kI$ are invertible for all integers $k\geq 0$. Then the following   finite matrix summation formulas of   Kamp\'e de F\'eriet hypergeometric matrix function hold true:
\begin{align}
&\sum_{k=0}^{r}{r\choose k}{(-1)^{k}F^{m_1; n_1, n'_1}_{m_2; n_2, n'_2}\left(^{A: B, C}_{D: E, F_{i}-kI, F^{i}}; x, y\right)(I-F_{i})_{k}}{((2-r)I-F_{i})^{-1}_{k}}\notag\\
&={(-1)^{r}[A]_{r}[C]_{r}}\,y^{r} F^{m_1; n_1, n'_1}_{m_2; n_2, n'_2}\left(^{A+rI: B, C+rI}_{D+rI: E, F+rI}; x, y\right){(F_{i}-I)^{-1}_{r}[D]^{-1}_{r}[F]^{-1}_{r}};\label{s4}
\end{align}
\begin{align}
&\sum_{k=0}^{r}{r\choose k}{(-1)^{k}F^{m_1; n_1, n'_1}_{m_2; n_2, n'_2}\left(^{A: B, C}_{D: E, F_{i}+kI, F^{i}}; x, y\right)(F_{i}+(r-1)I)_{k}}{(F_{i})^{-1}_{k}}\notag\\
&={[A]_{r}[C]_{r}}\,y^{r}\,F^{m_1; n_1, n'_1}_{m_2; n_2, n'_2}\left(^{A+rI: B, C+rI}_{D+rI: E, F_{i}+2rI, F^{i}+rI}; x, y\right){[D]^{-1}_{r}[F]^{-1}_{r}(F_{i}+rI)^{-1}_{r}},\label{s5}
\end{align}
where $(2+k-r)I-F_{i}$ , $F_i-kI$ and $F_i+(k-1)I$ is an invertible matrix for $0\leq k\leq r$  in (\ref{s4});   $F_i+rI$  is an invertible matrix in (\ref{s5});  $i=1,\dots,n'_{2}.$
\end{theorem}
{\bf Proof:} We first prove identity (\ref{s4}). 
From the definition of  Kamp\'e de F\'eriet hypergeometric matrix function and  the generalized Leibnitz formula for differentiation of a product of two functions, we obtain the following result:
\begin{align*}
&D_{y}^{r}\{F^{m_1; n_1, n'_1}_{m_2; n_2, n'_2}\left(^{A: B, C}_{D: E, F}; x, y\right) y^{I-F_{i}}\times y^{F_{i}-I}\}\\
&=\sum_{k=0}^{r}{r\choose k}D_{y}^{k}\{\,F^{m_1; n_1, n'_1}_{m_2; n_2, n'_2}\left(^{A: B, C}_{D: E, F}; x, y\right)y^{F_{i}-I}\}D_{y}^{r-k}\{y^{I-F_{i}}\}\\
&=\sum_{k=0}^{r}(-1)^{r+k}{r\choose k}
 F^{m_1; n_1, n'_1}_{m_2; n_2, n'_2}\left(^{A: B, C}_{D: E, F_{i}-kI, F^{i}}; x, y\right)\\&\quad\times{(F_{i}-I)_{r}(I-F_{i})_{k}}{(2I-F_{i}-rI)^{-1}_{k}\,y^{-r}}.
\end{align*}
Now using the derivative operator on Kamp\'e de F\'eriet hypergeometric matrix function for $r$-times directly and equating with the above equality gives (\ref{s4}) after some simplification. Next, applying the operator $D_{y}^{r}$ on 
\begin{align*}
F^{m_1; n_1, n'_1}_{m_2; n_2, n'_2}\left(^{A: B, C}_{D: E, F_{i}+rI, F^{i}}; x, y\right)y^{I-F_{i}-rI}\times y^{F_{i}+Ir-I},
\end{align*}
and proceeding as in the proof of (\ref{s4}) gives the result (\ref{s5}).

\section{Infinite summation formulas for Kamp\'e de F\'eriet hypergeometric matrix function}
In this section, we will  establish the infinite summation formulas of Kamp\'e de F\'eriet hypergeometric matrix function.
\begin{theorem}\label{t1}The following infinite summation formulas of Kamp\'e de F\'eriet hypergeometric matrix function hold true:
\begin{align}
&\sum_{k=0}^{\infty}\frac{(A_{i})_{k}}{k!}\,t^{k}F^{m_1; n_1, n'_1}_{m_2; n_2, n'_2}\left(^{A_i+kI, \,A^{i}: B, \,C}_{\,\,\,\,\,\,D\,\,\,\,\,\,\,\,\,\,\,: \,E, \,\,\,\,\,F\,\,\,}; x, y\right)\notag\\
&=\, (1-t)^{-A_{i}}\,F^{m_1; n_1, n'_1}_{m_2; n_2, n'_2}\left(^{A: B, C}_{D: E, F}; \frac{x}{1-t}, \frac{y}{1-t}\right),\label{1s2}
\end{align}
where $i=1,\dots,m_1$;
\begin{align}
&\sum_{k=0}^{\infty}\frac{(B_{i})_{k}}{k!}\,t^{k}F^{m_1; n_1, n'_1}_{m_2; n_2, n'_2}\left(^{A : B_i+kI,\, B^{i},  C}_{ D : \,\,\,E,\,\, F}; x, y\right)\notag\\
&=\, (1-t)^{-B_{i}}\,F^{m_1; n_1, n'_1}_{m_2; n_2, n'_2}\left(^{A: B, C}_{D: E, F}; \frac{x}{1-t}, {y}\right),  A_i B_j=B_j A_i\label{t1s2}
\end{align}
where  $i=1,\dots,n_1$.
\end{theorem}
{\bf Proof:} We give a proof of identity (\ref{1s2}). Applying the definition of  Kamp\'e de F\'eriet hypergeometric matrix function  and  transformation \begin{align*}(A_{i})_{k} \,(A_{i}+kI)_{m+n}= (A_{i})_{m+n}\, (A_{i}+(m+n)I)_{k},\end{align*} the left side of (\ref{1s2}) can be expressed as  
\begin{align*}
\sum_{m, n=0}^{\infty} \, _{1}F_{0}\left[^{A_{i}+(m+n)I}_{\,\,\,\,\quad-\,\,\,\,}; t\right]\prod_{i=1}^{m_1}(A_i)_{m+n} \prod_{i=1}^{n_1}(B_i)_{m} \prod_{i=1}^{n'_1}(C_i)_{n} \prod_{i=1}^{m_2}(D_i)^{-1}_{m+n} \prod_{i=1}^{n_2}(E_i)^{-1}_{m} \prod_{i=1}^{n'_2}(F_i)^{-1}_{n}\frac{x^{m} y^{n} }{m! n! }.
\end{align*}
Using the identity
\begin{align}
_{1}F_{0}\left[{A};-; t\right]= (1-t)^{-A}.\label{N}
\end{align}
After some simplification, we get the right side of (\ref{1s2}). This completes the proof of (\ref{1s2}). Identitiy (\ref{t1s2})  are proved in a similar manner.

%\begin{corollary}%Let  $A$, $B$, $B'$, $C$  and  $C'$ be commuting matrices.
% The following infinite summation formulas of Appell matrix  functions  hold true:
%\begin{align*}
%&\sum_{k=0}^{\infty}\frac{(A)_{k}}{k!}\, t^{k}\, F_{1}\left({A}+kI, B, B'; C ;   x, y\right)\notag\\
%&=(1-t)^{-A}\,F_{1}\left({A}, B, B'; C ; \frac{x}{1-t},\frac{y}{1-t}\right);
%\end{align*}
%\begin{align*}
%&\sum_{k=0}^{\infty}\frac{(A)_{k}}{k!}\, t^{k}\, F_{4}\left({A}+kI, B;C, C'; x, y\right)\notag\\
%&=(1-t)^{-A}\,F_{4}\left(A, B; C, C';\frac{x}{1-t},\frac{y}{1-t}\right).
%\end{align*}
%\end{corollary}
\begin{theorem}
Let  $A_iB_j=B_jA_i$;  $D_i E_j= E_j D_i$; $D_iF_j= F_j D_i$; $E_i F_j=F_j E_i$. Then  infinite summation formulas of Kamp\'e de F\'eriet hypergeometric matrix function hold true:
\begin{align}
F^{m_1; n_1, n'_1}_{m_2; n_2, n'_2}\left(^{A: B, C}_{ D : E, F}; x+t, y\right)=&\sum_{k=0}^{\infty}{[A]_{k}[B]_{k}}\,\frac{t^{k}}{k!}\notag\\
\qquad\qquad&\times F^{m_1; n_1, n'_1}_{m_2; n_2, n'_2}\left(^{A+kI: B+kI, C}_{ D+kI : E+kI, F}; x, y\right)[D]^{-1}_{k}[E]^{-1}_{k}.\label{s3}
\end{align}
%\begin{align}
%F^{m_1; n_1, n'_1}_{m_2; n_2, n'_2}\left(^{A: B, C}_{ D : E, F}; x, y+t\right)=&\sum_{k=0}^{\infty}{[A]_{k}[C]_{k}}\,\frac{t^{k}}{k!}\notag\\
%\qquad\qquad&\times F^{m_1; n_1, n'_1}_{m_2; n_2, n'_2}\left(^{A+kI: B, C+kI}_{ D+kI : E, F+kI}; x, y\right)[D]^{-1}_{k}[F]^{-1}_{k}.\label{t1s3}
%\end{align}

\end{theorem}
{\bf Proof:} From the definition Kamp\'e de F\'eriet hypergeometric matrix function and the transformation $(A)_{k}\,(A+kI)_{m}= (A)_{k+m}$, the right side of (\ref{s3}) can be written as 
\begin{align*}
\sum_{k, m, n=0}^{\infty}\prod_{i=1}^{m_1}(A_i)_{m+n+k} \prod_{i=1}^{n_1}(B_i)_{m+k} \prod_{i=1}^{n'_1}(C_i)_{n} \prod_{i=1}^{m_2}(D_i)^{-1}_{m+n+k} \prod_{i=1}^{n_2}(E_i)^{-1}_{m+k} \prod_{i=1}^{n'_2}(F_i)^{-1}_{n}\frac{x^{m} y^{n} }{m! n! k!}\, t^{k}.
\end{align*}
Replacing $m+k\rightarrow l$ in the above result and after some simplification, we get
\begin{align*}
&\sum_{l, n=0}^{\infty}\prod_{i=1}^{m_1}(A_i)_{l+n} \prod_{i=1}^{n_1}(B_i)_{l} \prod_{i=1}^{n'_1}(C_i)_{n} \prod_{i=1}^{m_2}(D_i)^{-1}_{l+n} \prod_{i=1}^{n_2}(E_i)^{-1}_{l} \prod_{i=1}^{n'_2}(F_i)^{-1}_{n}\frac{ y^{n} }{l! n! }\, \sum_{k=0}^{l}{l\choose k}x^{l-k} t^{k}\notag\\
&=\sum_{l, n=0}^{\infty}\prod_{i=1}^{m_1}(A_i)_{l+n} \prod_{i=1}^{n_1}(B_i)_{l} \prod_{i=1}^{n'_1}(C_i)_{n} \prod_{i=1}^{m_2}(D_i)^{-1}_{l+n} \prod_{i=1}^{n_2}(E_i)^{-1}_{l} \prod_{i=1}^{n'_2}(F_i)^{-1}_{n}\frac{(x+t)^{l} y^{n} }{l! n! }.
\end{align*}
Using relation
\begin{align}
\sum_{k=0}^{l}{l\choose k} x^{k} y^{l-k}= (x+y)^{l}
\end{align}
in the inner summation. This completes the proof of (\ref{s3}). % Identitiy (\ref{t1s3}) can be proved in an analogous manner.\\

%By specializing the parameters in the above theorem, we can obtain infinite summation formulas of Appell matrix functions. 
%We have
%\begin{corollary} The  following infinite summation formulas of Appell matrix  functions  hold true:
%\begin{align*}
%&\sum_{k=0}^{\infty}(A)_{k}(B)_{k}\, \frac{t^{k}}{k!}\, F_{1}\left({A}+kI, B+kI, B'; C+kI ;   x, y\right){(C)^{-1}_{k}\,}\notag\\
%&=\,F_{1}\left({A}, B, B'; C ;   x+t, y\right),
%\end{align*}
%where $AB=BA$;
%\begin{align*}
%&\sum_{k=0}^{\infty}{(A)_{k}(B)_{k}}\, \frac{t^{k}}{k!}\, F_{3}\left({A}+kI, A', {B+kI}, B'; C+kI ; x,  y\right){(C)^{-1}_{k}}\notag\\
%&=\,F_{3}\left(A, A', B, B'; C; x+t, y\right),
%\end{align*}
%where $AB=BA$; $A'B=BA'$.
%\end{corollary}
\begin{theorem}\label{t9}
The following infinite summation formulas of Kamp\'e de F\'eriet hypergeometric matrix function hold true:
\begin{align}
&\sum_{k=0}^{\infty}\frac{(B_{i})_{k}}{k!}\,(-t)^{k}F^{m_1; n_1, n'_1}_{m_2; n_2, n'_2}\left(^{A: -kI, B^{i}, C}_{ D : E, F}; \frac{1+t}{t}x, y\right)\notag\\
&=\, (1+t)^{-B_{i}}\,F^{m_1; n_1, n'_1}_{m_2; n_2, n'_2}\left(^{A: B, C}_{ D : E, F}; x, y\right),\label{10s1}
\end{align}
where  $A_j B_i = B_i A_j$,  $i=1,\dots,n_1$.
\end{theorem}
{\bf Proof:} From the definition of Kamp\'e de F\'eriet hypergeometric matrix function, the left side of (\ref{10s1}) can be expressed as
\begin{align*}
&\sum_{k,n =0}^{\infty}\sum_{m=0}^{k}\frac{(B_{i})_{k}(-kI)_{m}}{m! n!  k!}\prod_{i=1}^{m_1}(A_i)_{m+n} \prod_{j=1,\neq i}^{n_1}(B_i)_{m} \prod_{i=1}^{n'_1}(C_i)_{n} \prod_{i=1}^{m_2}(D_i)^{-1}_{m+n} \prod_{i=1}^{n_2}(E_i)^{-1}_{m} \prod_{i=1}^{n'_2}(F_i)^{-1}_{n}\notag\\&(\frac{1+t}{t} x)^{m} y^{n} (-t)^{k},
\end{align*}
Replacing $k=m+l$, changing the summation order and simplifying, we get
\begin{align*}
&\sum_{m, n=0}^{\infty}\, _{1}F_{0}\left[^{B_{i}+mI}_{\,-\,}; -t\right]\prod_{i=1}^{m_1}(A_i)_{m+n} \prod_{i=1}^{n_1}(B_i)_{m} \prod_{i=1}^{n'_1}(C_i)_{n} \prod_{i=1}^{m_2}(D_i)^{-1}_{m+n} \prod_{i=1}^{n_2}(E_i)^{-1}_{m} \prod_{i=1}^{n'_2}(F_i)^{-1}_{n}\frac{x^{m}y^{n}}{m! n! } \notag\\& (1+t)^{m} .
\end{align*}
Evaluating the inner $ _{1}F_{0}$-series in the above equation by (\ref{N}) 
\begin{align*}
_{1}F_{0}\left[{B_{i}+mI}; - ; -t\right]= (1+t)^{-B_{i}-mI}.
\end{align*}
and simplifying, we get the right side of this theorem. This completes the proof of this theorem.

%\begin{corollary} The following infinite summation formulas of Appell matrix  functions  hold true:
%\begin{align*}
%&\sum_{k=0}^{\infty}\frac{(B)_{k}}{k!}(-t)^{k}\,F_{1}\left(A,-kI, B'; C; \frac{1+t}{t}x, y\right)\notag\\
%&=(1+t)^{-B}\,F_{1}\left(A, B, B'; C; x, y\right);
%\end{align*}
%\begin{align*}
%&\sum_{k=0}^{\infty}\frac{(B)_{k}}{k!}(-t)^{k}\,F_{2}\left(A, -kI, B';C, C'; \frac{1+t}{t}x, y\right)\notag\\
%&=(1+t)^{-B}\,F_{2}\left(A, B, B'; C, C' ; x, y\right),	
%\end{align*}
%where   $AB=BA$.
%\end{corollary}

\begin{theorem}Let $A_i$, $B_i$, $C_i$, $D_i$, $E_i$, $F_i$ are positive stable matrices in $\mathbb{C}^{r\times r}$. Then 
 infinite summation formulas of Kamp\'e de F\'eriet hypergeometric matrix function hold true:
\begin{align}
&\sum_{k=0}^{\infty}\frac{(B_{i})_{k}}{k!}\,\Big(\frac{t+x}{x-1}\Big)^{k}\,F^{m_1; n_1, n'_1}_{m_2; n_2, n'_2}\left(^{A: -kI, B^{i}, C}_{ D : E, F}; \frac{1+t}{t+x}x, y\right)\notag\\
&=\,\Big (\frac{1-x}{t+1}\Big)^{B_{i}}\,F^{m_1; n_1, n'_1}_{m_2; n_2, n'_2}\left(^{A: B, C}_{ D : E, F}; x, y\right),\label{10s2}
\end{align}
where $A_j B_i= B_i A_j$, $i=1,\dots, n_1$.
\end{theorem}
{\bf Proof:} The proof of this theorem is similar to Theorem \ref{t9}. 
We omit the details.
%\begin{corollary}The following infinite summation formulas of Appell matrix  functions  hold true:
%\begin{align*}
%&\sum_{k=0}^{\infty}\frac{(B)_{k}}{k!}\Big(\frac{t+x_{1}}{x_{1}-1}\Big)^{k}\,F_{1}\left(A,-kI, B'; C; \frac{1+t}{t+x}{x},y\right)\notag\\
%&=\Big(\frac{1-x}{1+t}\Big)^{B}\,F_{1}\left(A, B, B'; C; x, y\right);
%\end{align*}
%\begin{align*}
%&\sum_{k=0}^{\infty}\frac{(A)_{k}}{k!}\Big(\frac{t+x}{x-1}\Big)^{k}\,F_{3}\left(-kI, A', B, B'; C;\frac{1+t}{t+x} {x},y\right)\notag\\
%&=\Big(\frac{1-x}{1+t}\Big)^{A}\,F_{3}\left(A, A', B, B'; C ; x, y\right);	
%\end{align*}
%\end{corollary}
% where  $AB=BA$.

\end{document}